\documentclass[12pt]{article}
\usepackage[utf8]{inputenc}
\usepackage[T1]{fontenc}
\usepackage{amsmath,amssymb,amsfonts,amsthm,amstext}
\usepackage{comment}
\usepackage{bm}
\usepackage[numbers]{natbib}
\usepackage{nicefrac} 
\usepackage{bbm}
\usepackage[dvipsnames]{xcolor}

\scrollmode

\numberwithin{equation}{section}

\newtheorem{teo}{Theorem}
\newtheorem{prop}{Proposition}[section]
\newtheorem{lem}{Lemma}[section]

\newcommand{\Var}{\text{Var}}

\usepackage{lmodern}
\usepackage{titling}
\title{{\textsc{\Large Central limit theorems for a driven particle in a random medium with mass aggregation}}}

\author{\textsc{\large Luiz Renato Fontes}\thanks{Partially supported by CNPq grant 311257/2014-3, and FAPESP grant 2017/10555-0.  Instituto de Matemática e Estat\'istica, Universidade de São Paulo, Rua do Matão 1010, Cidade Universitária, 05508-090 São Paulo SP, Brasil. Email: lrfontes@usp.br }\and \textsc{\large Pablo Almeida Gomes}\thanks{Partially
		supported by CAPES.  Instituto de Ciências Exatas, Universidade Federal de Minas Gerais, Av. Antônio Carlos 6627, Pampulha, 31270-901  Belo Horizonte MG, Brasil. Email: pabloag@ufmg.br} \and \textsc{\large Remy Sanchis}\thanks{Partially
		supported by CAPES, CNPq and FAPEMIG (Programa Pesquisador Mineiro).  Instituto de Ciências Exatas, Universidade Federal de Minas Gerais, Av. Antônio Carlos 6627, Pampulha, 31270-901  Belo Horizonte MG, Brasil. Email: rsanchis@mat.ufmg.br}}

\date{}

\begin{document}
	
	\maketitle
	
	\vspace{-1cm}
	
	\begin{abstract}
		We establish central limit theorems for the position and velocity of the charged particle in the mechanical particle model introduced in~\cite{FNV}.
	\end{abstract}
	
	\vspace{.5cm}
	\vspace{.5cm}
	
	
	\vspace{.5cm}
	
	


	\section{Introduction}
	
	We revisit the $1d$ mechanical particle model introduced in~\cite{FNV}, where we have a charged particle initially standing at the origin, subjected to an electric field, in an environment of initially standing neutral particles of unit mass. 
	Each neutral particle has randomly either an elastic nature or an inelastic nature. With the first kind of neutral particle, the charged particle collides in a totally elastic fashion. And the collisions of the charged particle with the second kind of neutral particle is totally inelastic. The neutral particles do not interact amongst themselves.
	Both kinds of neutral particles are initially randomly placed in space. 
	
	In~\cite{FNV}, a law of large numbers was proved for the instantaneous velocity of the charged particle. In this article, we derive central limit theorems for both the position and the instantaneous velocity of that particle, 
	in a sense completing the result of~\cite{FNV}; see Final Remarks of~\cite{FNV}.
	
	Our approch is similar to that of~\cite{FNV}, namely, we first prove CLT's for the corresponding objects of a modified process, where there are no recollisions. The results for the original process are established by showing that the differences between the actual and modified quantities are negligible in the relevant scales.

	
	\section{The Model and Results} 
	
	We consider a system of infinitely many point like particles in the non-negative real semi-axis $[0,\infty)$. At time $0$ the system is static, every particle has velocity 0. There is a distinguished particle of mass $2$ initially at the origin; we will call it the {\em tracer particle (t.p.)} (referred to before as the charged particle). 
	The remaining particles  (referred to before as neutral particles) have mass $1$.\footnote{The distinction of the initial mass of the t.p.~with respect to the other particles, absent in~\cite{FNV}, is for convenience only; any positive initial mass for the t.p.~would not change our results, but values 1 or below would require unimportant complications in our arguments.}
	Let $\{\xi_i\}_{i \in \mathbb{N}}$ denote a family of i.i.d.~positive random variables, with an absolutely continuous distribution, and finite mean $\mathbb{E}\xi_1 = \mu < \infty$, representing the initial interparticle distances. In this way, $S_i = \xi_1 + \cdots + \xi_i$ denotes the position of the $i$-th particle initially in front of the t.p.~at time 0. Moreover, given a parameter $p \in (0,1]$, and a family $\{\eta_i\}_{i \in \mathbb{N}}$ of i.i.d.~Bernoulli random variables with success probability $p$, 
	we say that the $i$-th particle is {\it sticky} if $\eta_i = 1$ and is {\it elastic} if $\eta_i = 0$.
	We assume $\{\xi_i\}_{i \in \mathbb{N}}$ and $\{\eta_i\}_{i \in \mathbb{N}}$ to be independent of one another.
	
	A constant positive force $F$ is turned on at time 0, and kept on. It acts solely on the tracer particle, producing in it an accelerated motion to the right. Collisions will thus take place in the system; we assume they occur only when involving the t.p., and suppose that all other particles do not interact among themselves. If at an instant $t>0$, the t.p.~collides with a sticky particle, then this is a perfectly inelastic collision, meaning that, upon collision,
	momentum is conserved and the energy of the two particle system is minimum, 
	which in turn means that the t.p.~incorporates the sticky particle, along with its mass, and the new velocity of the t.p.~becomes
	(immediately after time $t$)
	\begin{equation}\label{eq:1}
	V(t^+) = \frac{M_t}{M_t + 1}V_{t},
	\end{equation}
	where $V_t$ and $M_t$ are respectively the velocity and mass of t.p.~at time $t$. However, if the t.p.~collides with an elastic particle which is moving at velocity $v$ at the time of the collision, say $t$, then we have a perfectly elastic collision, where energy and momentum are preserved, and in this case, immediately after time $t$, the t.p.~and the elastic particle velocities become, respectively,
	\begin{eqnarray}\label{eq:2}
	V(t^+) &=& \frac{M_t - 1}{M_t + 1}V_t + \frac{2}{M_t + 1}v ~~~ \mbox{and} \nonumber\\\label{eqd}
	v' &=& \frac{2M_t}{M_t + 1}V_t - \frac{M_t - 1}{M_t + 1}v,
	\end{eqnarray}
	where $V_t$ and $M_t$ are as above.
	
	For $t \geq 0$, let $V_t$ and $Q_t$  denote the velocity and position of the t.p.~at time $t$, respectively. 
	As argued in~\cite{FNV}, the
	stochastic process $(V_t,Q_t)_{t \geq 0}$ is well defined --- see the discussion at the end of Section 2 of~\cite{FNV}; in particular there a.s.~are no multiple collisions or infinitely many recollisions in finite time intervals ---, and is determined by $\{\xi_i, \eta_i ~;~ i \in \mathbb{N}\}$. Therefore we consider the product sample space $\Omega = \left\{(0,\infty) \times \left\{0,1\right\} \right\}^{\mathbb{N}}$, and the usual product Borel $\sigma$-algebra, and the product probability measure 
	$\mathbb{P} := \prod_{i \geq 1} [\mathbb{P}_{\xi_i} \otimes \mathbb{P}_{\eta_i} ]$, where for $i \geq 1$, $\mathbb{P}_{\xi_i}$ and $\mathbb{P}_{\eta_i}$ denote the probability measures of $\xi_i$ and $\eta_i$. We will make repeatedly make use of the notation
	$$\bar\xi_i=\xi_i-\mu,\,\bar\eta_i=\eta_i-p.$$

%
%

	\medskip

	From [1], we know that $\mathbb{P}$-almost surely, the  velocity of the t.p.~converges to a(n explicit) limit. More precisely, 
	we have the following result.
	\begin{teo}\label{Th: 1}
		The stochastic process $(V_t , Q_t)_{t \geq 0}$ is such that
		\[ \lim_{t \to \infty} V_t = \sqrt{\frac{F\mu}{2-p} } ~~~ \mathbb{P}-\mbox{a.s.}\]
	\end{teo}
	
	From now on we denote the limit velocity 
	$\sqrt{F\mu/(2-p)}$ by $V_L$. The purpose of this paper is to show that the velocity $V_t$ and position $Q_t$ of the tracer particle satisfy central limit theorems. Our main results are as follows (where "$\Longrightarrow$" denotes convergence in distribution).
	\begin{teo}\label{CLTQ}
		Let $\Var(\xi_1) = \sigma^2 < \infty.$ Then, as $t\to\infty$, 
		\[ \frac{Q_t - tV_L}{\sqrt{t}} \Longrightarrow \mathcal{N}(0,\sigma_q^2),
		\]
		where $\sigma_q>0$.
	\end{teo}   	
	
	\begin{teo}\label{CLTV}
		Let $\Var(\xi_1) = \sigma^2 < \infty.$ Then, as $t\to\infty$,  
		\[ \sqrt{t} (V_t - V_L) \Longrightarrow \mathcal{N}(0,\sigma_v^2),
		\]
		where $\sigma_v>0$.
	\end{teo}

	
	\section{Central Limit Theorems in a Modified Process}  
	
	As mentioned in the Introduction, we first prove central limit theorem analogues of Theorems \ref{CLTQ} and \ref{CLTV} for a modified process in which, when an elastic particle collides with the t.p., the elastic particle is annihilated and disappears from the system, and the velocity of the t.p.~changes according to the formula \eqref{eq:2}, while collisions between the t.p.~and sticky particles remain as in the original model. We denote the modified stochastic process by 
	$(\bar{V}(t), \bar{Q}(t))_{t \geq 0}$,  where $\bar{V}(t)$ and $\bar{Q}(t)$ are respectively the velocity and position of the t.p.~in the modified system at time $t$. 
	
	In the modified model, for $i \geq 1$, the t.p.~collides with the $i$-th particle only in the initial position of the latter particle, given by $S_i$; let us denote the instant when that collision occurs by $\bar{t}_i$, i.e., $\bar{Q}(\bar{t}_i) = S_i$. In this way, we can compute the $i-$th collision incoming and outgoing velocities $\bar{V}(\bar{t}_i)$ and $\bar{V}(\bar{t}_i^+)$, respectively, as follows. First note that, according the formulas \eqref{eq:1} and \eqref{eq:2}, we have the following relations
	\begin{eqnarray*}
		\mbox{(a)} ~~~ \bar{V}^2(\bar{t}_i^{\phantom{+}}) &=& \bar{V}^2(\bar{t}_{i-1}^+) + \frac{2F\xi_i}{M(\bar{t}_i)}; \nonumber \\
		\mbox{(b)} ~~~ \bar{V}^2(\bar{t}_i^+) &=&  \bar{V}^2(\bar{t}_i) \left[ \frac{M(\bar{t}_i) + (\eta_i - 1)}{M(\bar{t}_i) + 1}\right]^2,
	\end{eqnarray*}
	where $M(\bar{t}_i) = 2+ \sum_{l=1}^{i-1} \eta_l$.
	
	Iterating this relations, we get for $i = 1, 2, \ldots$, that
	\begin{equation}\label{eq: DefV}
	\bar{V}^2(\bar{t}_i^+) = \sum_{j = 1}^i \left[ \frac{2F\xi_j}{M(\bar{t}_j)}\prod_{k=j}^i\left( \frac{M(\bar{t}_k) + (\eta_k - 1)}{M(\bar{t}_k) + 1}\right)^2 \right].
	\end{equation} 
	
	In [1], it is proved that, almost surely, 
	\begin{equation*}
	\lim_{t \to \infty} \bar{V}(t) = V_L.
	\end{equation*}
	
	Let us at this point set some notation. 
	Given two random sequences  $\{X_n\}_{n \in \mathbb{N}}$ and $\{Y_n\}_{n \in \mathbb{N}}$, we write $X_n = O(Y_n)$ if there  almost surely exists $C > 0$, which may be a (proper) random variabe, but does not depend on $n$, such that $|X_n| \leq CY_n$ for every $n \in \mathbb{N}$. And we say $X_n = o(Y_n)$ if $X_n/Y_n$ almost surely converges to $0$ as $n \to\infty$.
	For simplicity, along the rest of the paper we denote $M(\bar{t}_i)$  by $M_i$. Notice that $M_1 = 2$ and 
	$M_i = 2 + \sum_{k=1}^{i-1} \eta_k$, $i\geq2$.
	
	To obtain the central limit theorems for the modified process, we start with an estimate for the random term 
	\begin{equation}\label{eq: X}
	X_{i,j} := \frac{1}{M_j}\prod_{k=j}^i\left( \frac{M_k + (\eta_k - 1)}{M_k + 1}\right)^2, ~~~  1 \leq j \leq i ~ \text{and} ~ i \in \mathbb{N}.
	\end{equation}
	Given $\varepsilon > 0$, for each $m \in \mathbb{N}$ we define the event 
	\begin{equation}\label{eq: A}
	A_{m,\varepsilon} = \left\{ X_{i,j} \in \left((1-\varepsilon)\frac{j^{\zeta - 1}}{p i^{\zeta}}, (1+\varepsilon)\frac{j^{\zeta - 1}}{p i^{\zeta}}  \right), ~~~ \forall (i,j) ~ \text{such that} ~ m \leq j \leq i \right\},
	\end{equation}
	where $\zeta := 2(2-p)/p$. 
	\begin{lem}\label{Le: 1}
		Let $X_{i,j}$ be as in \eqref{eq: X}, and $A_{m, \varepsilon}$ as in \eqref{eq: A}, where $\varepsilon > 0$ is otherwise arbitrary. Then we have that 
		 \[\lim_{m \to \infty} \mathbb{P}\left(A_{m,\varepsilon}\right) = 1. \]
	\end{lem}\label{Th: L1}
	\begin{proof}
		We first  Taylor-expand the logarithm to write 
		\begin{eqnarray}\label{eq: C1}
		\prod_{k=j}^{i}\left( \frac{M_k + (\eta_k - 1)}{M_k + 1}\right)^2 &=& \exp \left\{ 2\sum_{k=j}^{i} \log \left(1 - \frac{2- \eta_k}{M_k + 1}  \right) \right\} \nonumber \\ &=& \exp \left\{- 2 \sum_{k=j}^{i} \left[  \frac{2-p}{M_k + 1} - \frac{\bar\eta_k}{M_k + 1} \right] + O\left( \sum_{k=j}^{i} \left( \frac{2-\eta_k}{M_k + 1}\right)^2 \right) \right\}. 
		\end{eqnarray}  
		Given $\delta > 0$, $m \in \mathbb{N}$, let $B_m^{\delta} = \left\{ M_j \in \left((1-\delta)pj, (1+\delta)pj\right), ~\forall j \geq m  \right\}	$. 
		It follows from the Law of Large Numbers that $P(B_m^{\delta})\to1$ a.s.~as $m\to\infty$. 
		In $B_m^{\delta}$, we have
		\begin{equation}\label{eq: C2}
		\sum_{k=1}^{\infty} \left( \frac{2-\eta_k}{M_k + 1}\right)^2 \leq \sum_{k=1}^{m-1} \left( \frac{2-\eta_k}{M_k + 1}\right)^2 + \frac{4}{p^2(1-\delta)^2}  \sum_{k=m}^{\infty} \frac{1}{k^2} < \infty.
		\end{equation}
		Note also that
		\begin{equation}\label{eq: 1/M}
		\sum_{k=j}^{i} \frac{1}{M_k + 1} = \sum_{k=j}^{i} \left( \frac{1}{M_k + 1} - \frac{1}{pk}\right) + \frac{1}{p}\left[ \sum_{k=j}^{i}  \frac{1}{k} -  \int_{j}^{i} \frac{1}{x} dx  \right] + \frac{1}{p} \int_{j}^{i} \frac{1}{x} dx.
		\end{equation}  
		
		Clearly the second term at the right-hand side of \eqref{eq: 1/M} goes to $0$ as $j$ and $i$ goes to infinity. 
		Let now $C_m= \left\{ |M_j + 1 - jp|  \leq j^{2/3}, ~\forall j \geq m \right\}.$ It follows from Law of the Iterated Logarithm that $\lim_{m \to \infty}\mathbb{P}(C_m) = 1$. In $B_m^{\delta} \cap C_m$ we have
		\begin{eqnarray}\label{eq: Mod}
		\left| \sum_{k=1}^{\infty} \left( \frac{1}{M_k + 1} - \frac{1}{pk}\right) \right| 
		&\leq& \left| \sum_{k=1}^{m-1} \left( \frac{1}{M_k + 1} - \frac{1}{pk}\right) \right| 
		+ \frac{1}{p(1-\delta)}\sum_{k=m}^{\infty} \frac{|M_k + 1 - kp|}{k^2} \nonumber \\
		&\leq& \left| \sum_{k=1}^{m-1} \left( \frac{1}{M_k + 1} - \frac{1}{pk}\right) \right| 
		+ \sum_{k=m}^{\infty} \frac{1}{k^{4/3}} < \infty.
		\end{eqnarray} 
		
		We also write 
		\begin{equation}\label{eq: TwoS}
		\sum_{k=1}^{\infty} \frac{\bar\eta_k}{M_k + 1} = \sum_{k=1}^{\infty}\left[ \bar\eta_k\left( \frac{1}{M_k + 1} - \frac{1}{pk}\right) \right] + \sum_{k=1}^{\infty} \frac{\bar\eta_k}{pk}.
		\end{equation}
		We may apply Kolmogorov's Two-series Theorem to obtain that $\sum_{k=1}^{\infty} \bar\eta_k/k$ converges a.s., and proceeding as in the estimation leading to \eqref{eq: Mod}, we may conclude that the first term in the right-hand side of \eqref{eq: TwoS} is also convergent in the event $B_m^{\delta} \cap C_m$.
		
		To conclude, due to \eqref{eq: C1}, \eqref{eq: C2}, \eqref{eq: 1/M}, \eqref{eq: Mod} and \eqref{eq: TwoS}, taking $\delta > 0$ sufficient small and $m$ sufficient large, we have that, in the event $B_m^{\delta} \cap C_m$, 
		\begin{equation}\label{eq:C3}
		\prod_{k=j}^{i}\left( \frac{M_k + (\eta_k - 1)}{M_k + 1}\right)^2 \in (1 \pm \varepsilon) \exp\left\{ -\zeta 
		\int_{j}^{i} \frac{1}{x} dx \right\}.
		\end{equation}
		
		Recalling now the definition of $X_{i,j}$ and $A_{m,\varepsilon}$ in \eqref{eq: X} and \eqref{eq: A}, respectively, 
		we have that \eqref{eq:C3} implies
		that $B_m^{\delta} \cap C_m \subset A_{m, \varepsilon}$, and the result follows.    
	\end{proof}	

\medskip
	
	We now turn our attention to $S_n - \bar{t}_nV_L$, for which we will prove a central limit theorem, as a step to establish Theorem~\ref{CLTQ}, as follows.
	
	\begin{prop}\label{Pr: TCL1}	
		Let $\Var(\xi_1) = \sigma^2 < \infty.$ Then, as $n\to\infty$,
		\begin{equation}\label{eq:clts}
		\frac{S_n - \bar{t}_nV_L}{\sqrt{n}} \Longrightarrow \mathcal{N}(0,\hat\sigma_q^2),
		\end{equation}
		where $\hat\sigma_q>0$.
	\end{prop}   
	The proof of this result consists of a number of steps which take most of this section.
	
	\smallskip
	
	From elementary physics relations, the time taken for the t.p.~to go from $S_{i-1}$ to $S_i$ is given by
	\begin{equation*}
	\bar{t}_i - \bar{t}_{i-1} = \frac{\bar{V}(\bar{t}_i) - \bar{V}(\bar{t}_{i-1}^{+}) }{F/M_i} 
	= \frac{2\xi_i \left( \bar{V}(\bar{t}_i) - \bar{V}(\bar{t}_{i-1}^{+}) \right)}{2\xi_i F / M_i} 
	= \frac{2\xi_i}{\bar{V}(\bar{t}_i) + \bar{V}(\bar{t}_{i-1}^{+})}.
	\end{equation*}
	Thus, we may write 
	\begin{eqnarray}\label{eq: Ap1}
	S_n - \bar{t}_nV_L &=& \sum_{i=1}^n \left[ \xi_i \left( 1 - \frac{2V_L}{\bar{V}(\bar{t}_i) + \bar{V}(\bar{t}_{i-1}^{+})}\right) \right] \nonumber \\
	&=&
	\sum_{i=1}^n \left[ \xi_i \left( \frac{\bar{V}(\bar{t}_i) + \bar{V}(\bar{t}_{i-1}^{+}) - 2V_L}{\bar{V}(\bar{t}_i) + \bar{V}(\bar{t}_{i-1}^{+})} \right) \right] \nonumber \\
	&=& 
	\sum_{i=1}^n \left[ \frac{2\xi_i \left( \bar{V}(\bar{t}_{i-1}^{+}) - V_L\right)}{\bar{V}(\bar{t}_i) + \bar{V}(\bar{t}_{i-1}^{+})}  \right]
	 +
	   \sum_{i=1}^n \left[ \xi_i \left( \frac{\bar{V}(\bar{t}_i) - \bar{V}(\bar{t}_{i-1}^{+})}{\bar{V}(\bar{t}_i) + \bar{V}(\bar{t}_{i-1}^{+})} \right) \right].
	\end{eqnarray}
	Note that 
	\begin{equation}\label{eq: Ap2}
	\frac{\bar{V}(\bar{t}_i) - \bar{V}(\bar{t}_{i-1}^{+})}{\bar{V}(\bar{t}_i) + \bar{V}(\bar{t}_{i-1}^{+})} = \frac{2F\xi_i}{M_i \left( \bar{V}(\bar{t}_i) + \bar{V}(\bar{t}_{i-1}^{+}) \right)^2 }.
	\end{equation}
	Since $\bar{V}(\bar{t}_i) + \bar{V}(\bar{t}_{i-1}^{+})$ converges to the constant $2V_L$,  the Law of Large Numbers and \eqref{eq: Ap2} imply that
	\begin{equation}\label{eq: Ap3}
	\sum_{i=1}^n \left[ \xi_i \left( \frac{\bar{V}(\bar{t}_i) - \bar{V}(\bar{t}_{i-1}^{+})}{\bar{V}(\bar{t}_i) + \bar{V}(\bar{t}_{i-1}^{+})} \right) \right] = 
	O \left( \sum_{i=1}^n \frac{\xi_i^2}{i}\right).
	\end{equation} 
	Let $\tilde{S}_0 = 0$ and 
	$\tilde{S}_k = \sum_{i=1}^k \xi_i^2$, $k\geq1$. Assuming $\mathbb{E}\xi_1^2 < \infty$, we have that
	\begin{equation}\label{eq: Ap4}
	\frac{1}{\sqrt{n}} \sum_{i=1}^n \frac{\xi_i^2}{i} = 
	\frac{1}{\sqrt{n}} \sum_{i=1}^n \frac{\tilde{S}_i - \tilde{S}_{i-1}}{i} = 
	\frac{1}{\sqrt{n}} \sum_{i=1}^{n-1} \frac{\tilde{S}_i}{i(i+1)} + \frac{\tilde{S}_n}{n^{3/2}} = o(1).
	\end{equation}
	Noticing that $\bar{V}(\bar{t}_i) = \bar{V}(\bar{t}_{i-1}^+) + 2F\xi_{i}\left(\bar{V}(\bar{t}_i) + \bar{V}(\bar{t}_{i-1}^+)\right)/M_i$, we find that
	\begin{eqnarray}
	 \frac{\bar{V}(\bar{t}_{i-1}^{+}) - V_L}{\bar{V}(\bar{t}_i) + \bar{V}(\bar{t}_{i-1}^{+})}  &=&  \frac{\bar{V}(\bar{t}_{i-1}^{+}) - V_L}{2V_L}
	  + 
	 \left[ \frac{\bar{V}(\bar{t}_{i-1}^{+}) - V_L}{\bar{V}(\bar{t}_i) + \bar{V}(\bar{t}_{i-1}^{+})} 
	  -
	   \frac{\bar{V}(\bar{t}_{i-1}^{+}) - V_L}{2V_L} \right] \nonumber \\
	 &=&
	 \frac{\bar{V}(\bar{t}_{i-1}^{+}) - V_L}{2V_L}
	 +
	 \frac{\left(\bar{V}(\bar{t}_{i-1}^{+}) - V_L\right)\left(2V_L - 2\bar{V}(\bar{t}_{i-1}^{+})\right)}{2V_L\left( \bar{V}(\bar{t}_i) + \bar{V}(\bar{t}_{i-1}^{+}) \right)}  
	 - 
	 \frac{2F\xi_i}{2V_LM_i}.
	\end{eqnarray}
	In particular, 
	\begin{multline}
	\sum_{i=1}^n \left[ \frac{2\xi_i \left( \bar{V}(\bar{t}_{i-1}^{+}) - V_L\right)}{\bar{V}(\bar{t}_i) + \bar{V}(\bar{t}_{i-1}^{+})}  \right]
	 = \\
	\sum_{i=1}^n \left[\frac{\xi_i \left(\bar{V}(\bar{t}_{i-1}^{+}) - V_L \right)}{V_L}   \right] 
	+
	 O\left( \sum_{i=1}^{n} \left[ \xi_i \left(\bar{V}(\bar{t}_{i-1}^{+}) - V_L\right)^2 \right] \right)
	 +
	 O \left(\sum_{i=1}^n \frac{\xi_i^2}{i} \right).
	\end{multline}
	Proceeding in an analogous way, we obtain that
	\begin{equation}\label{eq: Ap5}
		\sum_{i=1}^n \left[\frac{\xi_i \left(\bar{V}(\bar{t}_{i-1}^{+}) - V_L \right)}{V_L}   \right] 
		=
		 \sum_{i=1}^n \left[\frac{\xi_i \left(\bar{V}(\bar{t}_{i-1}^{+})^2 - V_L^2 \right)}{2V_L^2} \right]
		 + 
		  O\left( \sum_{i=1}^{n} \left[ \xi_i \left(\bar{V}(\bar{t}_{i-1}^{+}) - V_L\right)^2 \right] \right)   . 
	\end{equation}
	
	To simplify notation, for each $i \in \mathbb{N}$, we henceforth denote $\bar{V}(\bar{t}_{i}^{+})$ simply by $\bar{V}_i$. The following lemma will be useful now; we postpone its proof till the end of this section.
	\begin{lem}\label{Le: 2}
		Let $\Var(\xi_1) = \sigma^2 < \infty$ and let $\epsilon >0$. The velocities $\{\bar{V}_i\}_{i \in \mathbb{N}}$ are such that $\bar{V}_i - V_L = o(1/i^{1/2 - \epsilon})$. In particular,
		\[
		\frac{1}{\sqrt{n}} \sum_{i=1}^{n} \left[ \xi_i \left(\bar{V}_{i-1} - V_L\right)^2 \right] = o(1).
		\]
	\end{lem}
	
	By 
	\eqref{eq: Ap1} to \eqref{eq: Ap5} and Lemma \ref{Le: 2},
	in order to establish Proposition~\ref{Pr: TCL1} it
	is enough to show that as $n\to\infty$
	\begin{equation}\label{cltv2}
	\frac1{\sqrt n}\sum_{i=1}^{n} \xi_i (\bar{V}_{i-1}^{2} - V_L^{2}) \Longrightarrow \mathcal{N}(0,\tilde\sigma_q^2),
	\end{equation}
	for some $\tilde\sigma_q>0$; we then of course have $\hat\sigma_q=\tilde\sigma_q/(2V_L^2)$.
	For that, the strategy we will follow is to expand the expression on the left of~(\ref{cltv2})
	into several terms, one of which depends only on the interparticle distances $\{\xi_i\}_{i \in \mathbb{N}}$, another one depending only on the stickiness indicator random variables $\{\eta_k\}_{k \in \mathbb{N}}$; for each of those terms we can apply Lindeberg-Feller's Central Limit Theorem; upon showing that the remaining terms are negligible, the result follows. 
	
	Recalling that $\zeta = 2(2-p)/p$, \eqref{eq: DefV} and \eqref{eq: X}, we start with 
	\begin{multline}\label{eq:v21}
	\frac{1}{\sqrt{n}} \sum_{i=1}^{n} \left[ \xi_{i+1} (\bar{V}_{i}^{2} - V_L^{2}) \right] = \\
	\frac{2F}{\sqrt{n}} \sum_{i=1}^{n} \left[ \xi_{i+1}  \sum_{j=1}^{i}  \left(\xi_j X_{i,j} - \mu \frac{j^{\zeta-1}}{p i^{\zeta}}\right) \right] + 
	\frac{2F\mu}{p \sqrt{n}}\sum_{i=1}^{n} \left[ \xi_{i+1} \left( \frac{1}{i}\sum_{j=1}^{i} \left( \frac{j}{ i} \right)^{\zeta - 1} - \int_{0}^{1} x^{\zeta - 1} dx \right) \right]. 
	\end{multline}
	The term on the left of expression within parentheses in the second term on the right hand side of~(\ref{eq:v21}) is a Riemann sum for the term to its right; we conclude that the full expression within parenthesis on the right hand side of~(\ref{eq:v21}) is an $O(1/i)$,
	and we may thus conclude that the second term on the right-hand side of~(\ref{eq:v21}) is an $o(1)$, and proceed by dropping that term and focusing on the first term, which we write as follows.
	\begin{multline}\label{eq: VW}     
	\frac{2F}{\sqrt{n}} \sum_{i=1}^{n}  \left[ \xi_{i+1} \sum_{j=1}^{i}  \left(\xi_j X_{i,j} - \mu \frac{j^{\zeta-1}}{p i^{\zeta}}\right) \right] =\\
	\frac{2F}{\sqrt{n}} \sum_{i=1}^{n}  \left[ \bar\xi_{i+1} \sum_{j=1}^{i}  \left(\xi_j X_{i,j} - \mu \frac{j^{\zeta-1}}{p i^{\zeta}}\right) \right]  + 
	\frac{2F\mu}{\sqrt{n}} \sum_{i=1}^{n} \sum_{j=1}^{i}  \left(\xi_j X_{i,j} - \mu \frac{j^{\zeta-1}}{p i^{\zeta}}\right)\\  
	:=  V_{n} + W_{n}.
	\end{multline} 
    Now writing 
	\begin{equation*}
	\sum_{j=1}^{i}  \left(\xi_j X_{i,j} - \mu \frac{j^{\zeta-1}}{p i^{\zeta}}\right) = 
	\sum_{j=1}^{i}  \bar\xi_j X_{i,j}  +
	\mu \sum_{j=1}^{i} \left( X_{i,j} - \frac{j^{\zeta - 1}}{p i^{\zeta}} \right),
	\end{equation*} 
	$V_n$ given in \eqref{eq: VW} becomes 
	\begin{multline}\label{eq: V}
	V_n =  	\frac{2F}{\sqrt{n}} \sum_{i=1}^{n} \left[ \bar\xi_{i+1}\sum_{j=1}^{i} \left( \xi_j X_{i,j} - \mu \frac{j^{\zeta-1}}{p i^{\zeta}}\right) \right]  = \\
	\frac{2F}{\sqrt{n}} \sum_{i = 1}^{n} \sum_{j=1}^{i} \bar\xi_{i+1}\bar\xi_jX_{i,j} 
	+\frac{2F\mu}{\sqrt{n}} \sum_{i = 1}^{n} \left[ \bar\xi_{i+1} \sum_{j=1}^{i}  \left( X_{i,j}(\omega) - \frac{j^{\zeta - 1}}{p i^{\zeta}} \right) \right]     \\
	=: V_{1,n} + V_{2,n}.
	\end{multline}
	We will show in Lemmas~\ref{Le: Var} and~\ref{Le: Var1} below that $V_{1,n}$ and $V_{2,n}$ are negligible. 
	
	Analogously, $W_n$ given in \eqref{eq: VW} becomes 
	\begin{multline}\label{eq: W}
	W_n = \frac{2F\mu}{\sqrt{n}} \sum_{i=1}^{n} \sum_{j=1}^{i}  \left(\xi_j X_{i,j} - \mu \frac{j^{\zeta-1}}{p i^{\zeta}}\right) =\\
	\frac{2F\mu}{\sqrt{n}} \sum_{i = 1}^{n} \sum_{j=1}^{i} \bar\xi_j X_{i,j} 
	+\frac{2F\mu^2}{\sqrt{n}} \sum_{i = 1}^{n}  \sum_{j=1}^{i}  \left( X_{i,j} - \frac{j^{\zeta - 1}}{p i^{\zeta}} \right)  \\
	=: W_{1,n} + W_{2,n},
	\end{multline}
	and $W_{1,n}$ is further broken down into
	\begin{eqnarray}\label{eq: W4}
	W_{1,n} &=& \frac{2F\mu}{\sqrt{n}} \sum_{i = 1}^{n} \sum_{j=1}^{i} \bar\xi_j X_{i,j}  \nonumber\\
	&=& \frac{2F\mu}{p\sqrt{n}} \sum_{i = 1}^{n} \sum_{j=1}^{i}  \frac{j^{\zeta - 1}}{ i^{\zeta}} \bar\xi_j    +
	\frac{2F\mu}{\sqrt{n}} \sum_{i = 1}^{n} \sum_{j=1}^{i} \bar\xi_j \left( X_{i,j} - \frac{j^{\zeta - 1}}{p i^{\zeta}} \right)  \nonumber \\
	&=:& W_{3,n} + W_{4,n}.        
	\end{eqnarray}
	
	One may readily verify the conditions of Lindeberg-Feller's  CLT to obtain
	\begin{lem}\label{Le: L2} Let $\Var(\xi_1) = \sigma^2 < \infty$. For $1\leq j\leq n$, set $a_{j,n}=j^{\zeta - 1}\sum_{i=j}^{n} \frac1{i^\zeta}$. Then, as $n\to\infty$,
		\[ W_{3,n} = \frac{2F\mu}{p\sqrt{n}} \sum_{j=1}^{n}a_{j,n}\bar\xi_j 
		\Longrightarrow \mathcal{N}(0,\sigma_w^2), \]	
		where $\sigma_w=\frac{2F\mu}{p\sqrt\zeta}\,\sigma$.				
	\end{lem}
	In Lemma \ref{Le: Var2} below we show that $W_{4,n}$ is negligible. 
	
	Let us now focus on $W_{2,n}$. To alleviate notation, for each $1 \leq j \leq i$, set
	\begin{equation}\label{eq: Y}
	Y_{i,j} = \log \left[ \prod_{k=j}^{i}\left( \frac{M_k + (\eta_k - 1)}{M_k + 1}\right)^2 \right] = 2\sum_{k=j}^{i} \log \left(1 - \frac{2- \eta_k}{M_k + 1}  \right),    
	\end{equation}
	thus $X_{i,j} = e^{Y_{i,j}}/M_j$, and therefore,
	\begin{multline}\label{eq: Z}
	W_{2,n} = \frac{2F\mu^2}{\sqrt{n}} \sum_{i = 1}^{n}  \sum_{j=1}^{i}  \left( X_{i,j} - \frac{j^{\zeta - 1}}{p i^{\zeta}} \right) = \\
	\frac{2F\mu^2}{\sqrt{n}} \sum_{i = 1}^{n} \sum_{j=1}^{i} \left[ \left( \frac{1}{M_j} - \frac{1}{pj}\right) \frac{j^{\zeta}}{i^{\zeta}} \right] + 
	\frac{2F\mu^2}{\sqrt{n}} \sum_{i = 1}^{n}  \sum_{j=1}^{i} \left[ \left( \frac{1}{M_j} - \frac{1}{pj}\right) \left( e^{Y_{i,j}} - \frac{j^{\zeta}}{i^{\zeta}} \right) \right] + \\
	\frac{2F\mu^2}{\sqrt{n}} \sum_{i = 1}^{n}  \sum_{j=1}^{i} \left[\frac{1}{pj} \left( e^{Y_{i,j}} - \frac{j^{\zeta}}{i^{\zeta}} \right) \right] =: Z_{1,n} + Z_{2,n} + Z_{3,n}. 
	\end{multline}
	\begin{lem}\label{Le: L3}
		$Z_{2,n}$,  as defined in \eqref{eq: Z}, 
		is an $o(1)$.
	\end{lem}
	\begin{proof} 
		Note that, as defined in \eqref{eq: Y} and \eqref{eq: Z},
		\begin{eqnarray}\nonumber
			|Z_{2,n}| &=& \left|  \frac{2F\mu^2}{\sqrt{n}} \sum_{i = 1}^{n} \sum_{j=1}^{i} \left[ \left( \frac{1}{M_j} - \frac{1}{pj}\right) \left( e^{Y_{i,j}} - \frac{j^{\zeta}}{i^{\zeta}} \right) \right] \right| \\ \nonumber
			&=& 
			\left|  \frac{2F\mu^2}{\sqrt{n}} \sum_{i = 1}^{n} \sum_{j=1}^{i} \left[ \frac{j^{\zeta}}{i^{\zeta}} \left( \frac{1}{M_j} - \frac{1}{pj}\right) \left( \exp \left\{Y_{i,j} + \zeta \int_{j}^{i} \frac{1}{x} dx \right\} -  1 \right) \right] \right| \\
			\label{eq1}
			&\leq&  \frac{2F\mu^2}{\sqrt{n}} \sum_{i = 1}^{n} \sum_{j=1}^{i} \left[ \frac{j^{\zeta}}{i^{\zeta}} \left|  \frac{1}{M_j} - \frac{1}{pj}\right| \left| Y_{i,j} + \zeta \int_{j}^{i} \frac{1}{x} dx\right| \right].
		\end{eqnarray}	
		For each $i \geq j \geq 1$, we define 
		\begin{equation}\label{rij}
		R_{i,j} = Y_{i,j} +  \zeta \int_{j}^{i} x^{-1} dx. 	
		\end{equation}
		It follows from~(\ref{eq1}) that 
		\begin{equation}\label{eq: Z2n}
		|Z_{2,n}| = O\left( \frac{1}{\sqrt{n}} \sum_{i = 1}^{n} \sum_{j=1}^{i} \left[ \frac{j^{\zeta}}{i^{\zeta}} \left|  \frac{1}{M_j} - \frac{1}{pj}\right| \left| R_{i,j}\right| \right] \right).
		\end{equation}	
		
		As we see in \eqref{eq: C1} and \eqref{eq: Y}, $R_{i,j}$ can be written as
		\begin{multline}\label{eq: R}
		R_{i,j} = \zeta \int_{j}^{i} x^{-1} dx- 2 \sum_{k=j}^{i}   \frac{2-p}{M_k + 1}  + 
		2 \sum_{k=j}^{i}\frac{\bar\eta_k }{M_k + 1} + 
		O\left( \sum_{k=j}^{i} \left( \frac{2-\eta_k}{M_k + 1}\right)^2 \right) =  \\
		\zeta \left[  \int_{j}^{i} \frac{1}{x} dx  - \sum_{k=j}^{i} \frac{1}{k}    \right] + 
		\sum_{k=j}^{i} \left(  \frac{\zeta}{k} - \frac{2(2-p)}{M_k + 1}  \right) + 
		2 \sum_{k=j}^{i} \left[ \frac{\bar\eta_k }{M_k + 1} - \frac{\bar\eta_k }{p(k-1)+ 3} \right] + \\
		2\sum_{k=j}^{i} \frac{\bar\eta_k }{p(k-1)+3} +
		O\left( \sum_{k=j}^{i} \left( \frac{2-\eta_k}{M_k + 1}\right)^2 \right) 
		:= R^{(1)}_{i,j} + \cdots + R^{(5)}_{i,j}.
		\end{multline}
		%
		
		One readily checks by elementary deterministic estimation that
		for all $i \geq j\geq1$, $|R^{(1)}_{i,j}|$ can be bounded above by $1/j$. 
		
		Let now $0 < \delta < 1/4$ be fixed.  The Law of Large Numbers and the Law of the Iterated Logarithm, there a.s.~exists $j_0 \in \mathbb{N}$ such that $|R^{(2)}_{i,j}|$, $|R^{(3)}_{i,j}|$ and $|R^{(5)}_{i,j}|$ are bounded above by $1/j^{1/2 - \delta}$, for every $i \geq j \geq j_0$. 
		
		To study $|R^{(4)}_{i,j}|$, we apply Hoeffding's Inequality 
		to obtain, for every $i \geq j \geq 1$,
		\begin{equation}\label{eq: Ho1}
		\mathbb{P} \left( \left| \sum_{k=j}^{i} \frac{\bar\eta_k}{p(k-1)+3} \right| \geq \frac{1}{j^{1/2 - \delta}}\right) 
		\leq 
		\exp\left\{ - 2 \bigg/ \left( {j^{1-2\delta}\sum_{k=j}^{i} \frac{1}{(p(k-1) + 3)^2}} \right) \right\}.
		\end{equation} 
		We next apply a variation of Lévy's Maximal Inequality, namely Proposition 1.1.2 in \cite{DG}, 
		combined with \eqref{eq: Ho1},  to get  that
		\begin{eqnarray}\nonumber
			\mathbb{P} \left( \max_{i \geq j} \left| \sum_{k=j}^{i} \frac{\bar\eta_k }{p(k-1)+3} \right| \geq \frac{3}{j^{1/2 - \delta}} \right) &=&
			\lim_{n \to \infty} \mathbb{P} \left( \max_{j \leq i \leq n} \left| \sum_{k=j}^{i} \frac{\bar\eta_k }{p(k-1)+3} \right| \geq \frac{3}{j^{1/2 - \delta}} \right) \nonumber \\\nonumber
			&\leq& 3 \lim_{n \to \infty}  \max_{j \leq i \leq n} \mathbb{P} \left( \left| \sum_{k=j}^{i} \frac{\bar\eta_k }{p(k-1)+3} \right| \geq \frac{1}{j^{1/2 - \delta}} \right) \nonumber \\
			\label{eq:rij2}
			&\leq& 3 \exp\left\{ - 2 \bigg/ \left( {j^{1-2\delta}\sum_{k=j}^{\infty} \frac{1}{(p(k-1) + 3)^2}} \right) \right\}.
		\end{eqnarray}
		Since the latter term is summable, we conclude that almost surely exists $j_0 \in \mathbb{N}$ such that $|R^{(4)}_{i,j}| \leq {3}/{j^{1/2 - \delta}}$, for every $i \geq j \geq j_0$. 
		Collecting all the bounds, we find that a.s.
		\begin{equation}\label{eq:rij}
			|R_{i,j}| \leq |R^{(1)}_{i,j}| + \cdots + |R^{(5)}_{i,j}| < 3/j^{1/2 - \delta}
		\end{equation}
		for every $i \geq j$ sufficiently large. Recalling that $M_j = 2 + \sum_{l=1}^{j-1} \eta_l$, we have, as consequence of the Law of the Iterated Logarithm and the Law of Large Numbers, that $|1/M_j - 1/(pj)| = o(1/j^{3/2 - \delta})$, and 
		the result follows from \eqref{eq: Z2n}.
	\end{proof}

    It follows from~(\ref{eq:rij2}) that $R_{i,j}$ is uniformly bounded in $i,j$ by a proper random variable. We may thus write
	\begin{eqnarray}\label{eq: Z3}
	Z_{3,n} &=& \frac{2F\mu^2}{\sqrt{n}} \sum_{i = 1}^{n} \sum_{j=1}^{i} \left[  \frac{1}{pj} \left( e^{Y_{i,j}} - \frac{j^{\zeta}}{i^{\zeta}} \right) \right] = 
	\frac{2F\mu^2}{p\sqrt{n}} \sum_{i = 1}^{n} \sum_{j=1}^{i} \left[ \frac{j^{\zeta - 1}}{i^{\zeta}} \left( e^{R_{i,j}}
	- 1 \right) \right] \nonumber \\ &=& 
	\frac{2F\mu^2}{p\sqrt{n}} \sum_{i = 1}^{n} \sum_{j=1}^{i}\frac{j^{\zeta - 1}}{i^{\zeta}}  R_{i,j} + 
	O\left(\frac{2F\mu^2}{p\sqrt{n}} \sum_{i = 1}^{n} \sum_{j=1}^{i}\frac{j^{\zeta - 1}}{i^{\zeta}}R^2_{i,j} \right)
	=: Z_{3,n}' + \tilde Z_{3,n} . 
	\end{eqnarray}
	
	Since, almost surely, for every $i \geq j$ sufficiently large, we have the bound $|R^{(1)}_{i,j}| + |R^{(5)}_{i,j}| \leq 1/j^{2/3}$, it follows that
	\begin{equation*}
	\frac{2F\mu^2}{p\sqrt{n}}  \sum_{i = 1}^{n} \sum_{j=1}^{i} \left[  \frac{j^{\zeta - 1}}{i^{\zeta}} \left( R^{(1)}_{i,j} + R^{(5)}_{i,j} \right) \right] = o(1).
	\end{equation*}
	Considering only the term $R^{(2)}_{i,j}$ of $R_{i,j}$ in \eqref{eq: R}, its contribution to $Z_{3,n}'$ 
	in \eqref{eq: Z3} is
	\begin{multline}\label{eq:zn1}
	2(2-p)\frac{2F\mu^2}{p\sqrt{n}} \sum_{i = 1}^{n} \sum_{j=1}^{i} \left[ \frac{j^{\zeta - 1}}{i^{\zeta}}  
	\sum_{k=j}^{i} \left( \frac{1}{pk} -  \frac{1}{M_k+1} \right) \right] = \\ 
	\zeta\frac{2F\mu^2}{\sqrt{n}} \sum_{i = 1}^{n} \sum_{k=1}^{i} \left[ \left( \frac{1}{pk} -  \frac{1}{M_k+1} \right)   
	\sum_{j=1}^{k} \frac{j^{\zeta - 1}}{i^{\zeta}}  \right] = \\ 
	\frac{2F\mu^2}{\sqrt{n}} \sum_{i = 1}^{n} \sum_{k=1}^{i} \left[ \left( \frac{1}{pk} -  \frac{1}{M_k} \right) 
	\frac{k^{\zeta}}{i^{\zeta}} \right]  + o(1)= - Z_{1,n} + o(1),
	\end{multline}
	where $Z_{1,n}$ is defined in \eqref{eq: Z}. We may remark at this point that combining \eqref{eq:zn1} and \eqref{eq: Z} drops $Z_{1,n}$ out of the overall computation.
	
	Let us now estimate the contribution of $R^{(3)}_{i,j}$ to $Z_{3,n}'$ in \eqref{eq: Z3},
	recalling that $M_k = 2 + \sum_{l=1}^{k-1}\eta_l$ and setting $\bar{M}_k = -\sum_{l=1}^{k} \bar\eta_l $: 
	\begin{multline}\label{eq: Z5}
	\frac{4F\mu^2}{p\sqrt{n}}  \sum_{i = 1}^{n} \sum_{j=1}^{i} \left[  \frac{j^{\zeta - 1}}{i^{\zeta}} \sum_{k=j}^{i} 
	\left( \frac{\bar\eta_k }{M_k +1} - \frac{\bar\eta_k }{p(k-1)+ 3} \right)  \right] = \\ 
	\frac{4F\mu^2}{p\sqrt{n}}  \sum_{i = 1}^{n} \sum_{j=1}^{i} \left[  \frac{j^{\zeta - 1}}{i^{\zeta}} \sum_{k=j}^{i}  
	\frac{\bar\eta_k\bar{M}_{k-1}}{(M_k + 1)(p(k-1) + 3)}  \right] = 
	\frac{4F\mu^2}{p\sqrt{n}}  \sum_{i = 1}^{n} \sum_{j=1}^{i} \left[  \frac{j^{\zeta - 1}}{i^{\zeta}} \sum_{k=j}^{i}  
	\frac{\bar\eta_k \bar{M}_{k-1}}{(p(k-1)+3)^2}  \right]  + \\
	\frac{4F\mu^2}{p\sqrt{n}}  \sum_{i = 1}^{n} \sum_{j=1}^{i} \left[  \frac{j^{\zeta - 1}}{i^{\zeta}} \sum_{k=j}^{i} \left(  \frac{\bar\eta_k \bar{M}_{k-1}}{p(k-1) + 3}  \left( \frac{1}{M_k + 1} - \frac{1}{p(k-1)+ 3}\right) \right) \right] =: Z_{5,n} + Z_{6,n}. 
	\end{multline}

	Let us fix $0 < \alpha < 1/2$; the Law of the Iterated Logarithm and the Law of Large Numbers give us that 
	\begin{equation*}
	\left| \frac{\bar\eta_k \bar{M}_{k-1}}{p(k-1) + 3}  \left( \frac{1}{M_k + 1} - \frac{1}{p(k-1)+ 3}\right) \right| = o\left(\frac{1}{k^{2-\alpha}}\right). 
	\end{equation*}
	Since  $0 < \alpha < 1/2$, it follows that $Z_{6,n} = o(1)$. 
	
	We will study the asymptotic behavior of $Z_{5,n}$ in Lemma \ref{Le: Var3}. 
	
	We now estimate the contribution of $R^{(3)}_{i,j}$ to $Z_{3,n}'$ in \eqref{eq: Z3}:
%
	\begin{multline}\label{eq: Z4}
	\frac{4F\mu^2}{p\sqrt{n}} \sum_{i = 1}^{n} \sum_{j=1}^{i-1} \left[ \frac{j^{\zeta - 1}}{i^{\zeta}}  \sum_{k=j}^{i-1}   
	\frac{\bar\eta_k}{k} \right] 
	=  \frac{4F\mu^2}{p\sqrt{n}} \sum_{i = 1}^{n} \sum_{k=1}^{i-1} \left[\frac{\bar\eta_k}{k}   \sum_{j=1}^{k} \frac{j^{\zeta - 1}}{i^{\zeta}}     \right]  
	\\ = 
	\frac{4F\mu^2}{\zeta p\sqrt{n}} \sum_{i = 1}^{n} \sum_{k=1}^{i-1} \frac{k^{\zeta - 1}}{i^{\zeta}} \bar\eta_k     + o(1)  
	=: Z_{4,n} + o(1).
	\end{multline}
	By a routine verification of the conditions of the  Lindeberg-Feller CLT we get the following result.
	\begin{lem}\label{Le: L4}
		As $n\to\infty$ 
		\begin{equation*}
		Z_{4,n} = \frac{4F\mu^2}{\zeta p\sqrt{n}} \sum_{i = 1}^{n} \sum_{k=1}^{i-1} \frac{k^{\zeta - 1}}{i^{\zeta}} \bar\eta_k  
			 \Longrightarrow \mathcal{N}(0,\sigma_z^2),
		\end{equation*}
	where $\sigma_z=4F\mu^2\sqrt{\frac{1-p}{p\zeta^3}}$.
	\end{lem}
	
	Let us now estimate $\tilde Z_{3,n}$
	in \eqref{eq: Z3}. From \eqref{eq:rij} it readily follows that
	\begin{equation*}
	\frac{2F\mu^2}{p\sqrt{n}} \sum_{i = 1}^{n} \sum_{j=1}^{i}   \frac{j^{\zeta - 1}}{i^{\zeta}} R_{i,j}^2   = o(1),
	\end{equation*}
	and thus
	$\tilde Z_{3,n} = o(1)$.
	

\bigskip

    So far we have argued that
	\begin{eqnarray}\label{eq: GH}
	\frac{1}{\sqrt{n}} \sum_{i=1}^{n} \left[ \xi_{i+1}(\bar{V}_{i}^2 - V_L^2 ) \right] &=&	\big(W_{3,n} + Z_{4,n}\big) + \big(V_{1,n} + V_{2,n} + W_{4,n} + Z_{5,n}\big) + o(1) \nonumber \\  &=:& G_n + H_n + o(1),
	\end{eqnarray}
	where $W_{3,n}, W_{4,n}, Z_{4,n}, V_{1,n} , V_{2,n}$ and $Z_{5,n}$ are defined, respectively, in \eqref{eq: W4}, \eqref{eq: Z4}, \eqref{eq: V}  and \eqref{eq: Z5}. 
 By the independence of $W_{3,n}$ and $Z_{4,n}$, we have by Lemmas \ref{Le: L2} and \ref{Le: L4} that $G_n \Longrightarrow \mathcal{N}(0, \tilde\sigma_q^2)$, where $\tilde\sigma_q^2=\sigma_w^2+\sigma_z^2$. 
 To establish \eqref{cltv2}, it is then enough to show that $H_n=o(1)$, which we do in the following lemmas, one for each of the constituents of $H_n$.
	\begin{lem}\label{Le: Var}
			Assume $\Var(\xi_1) = \sigma^2 < \infty$. Then $V_{1,n}=o(1)$.
	\end{lem}
	\begin{proof}
		First fix $\delta > 0$. Given $\varepsilon > 0$, Lemma \ref{Le: 1} states that exists $m \in \mathbb{N}$ such that $\mathbb{P}(A_{m,\varepsilon}^c) < \varepsilon/2$. Recall the definition of $X_{i,j}$ in \eqref{eq: X}, 
		and that $\{X_{i,j}$, $i \geq j \geq 1\}$ and $\{\xi_n\}_{n \in \mathbb{N}}$ are independent. 
		\begin{equation}\label{eq: Var}
		\mathbb{P} \left( \left| \frac{1}{\sqrt{n}} \sum_{i = 1}^{n} \sum_{j=1}^{i} \bar\xi_{i+1}\bar\xi_{j}X_{i,j} \right| > \delta \right)  \leq 
		\mathbb{P} \left( \left| \frac{1}{\sqrt{n}} \sum_{i = 1}^{n} \sum_{j=1}^{i} \bar\xi_{i+1}\bar\xi_j X_{i,j}  1_{A_{m,\varepsilon}}\right| > \delta   \right) + \frac{\varepsilon}{2} 
		\end{equation}
		It follows from definition of $A_{m,\varepsilon}$ in \eqref{eq: A} that 
		 $X_{i,j}1_{A_{m,\varepsilon}}\leq (1+\varepsilon) [{j^{\zeta - 1}}/(p i^{\zeta})]$ for all $i\geq j \geq m$. Using this and  by Markov's Inequality, we get that the first term on the right of~(\ref{eq: Var}) is bounded above by
		\begin{eqnarray*}
			\frac{1}{\delta^2 n} \sum_{i=1}^{n} \sum_{j=1}^{i} \mathbb{E}(\bar\xi_{i+1})^2 \mathbb{E}(\bar\xi_j)^2 \, 
			\mathbb{E}(X_{i,j}^2 1_{A_{m,\varepsilon}})
			= \frac{(1+ \varepsilon)^2\sigma^4}{p^2\delta^2 n} \sum_{i = 1}^{n} \sum_{j=m}^{i} \frac{j^{2\zeta - 2}}{i^{2\zeta}} + o(1) = o(1).
		\end{eqnarray*}
		Since $\delta > 0$ and $\varepsilon > 0$ are arbitrary, the combination of this inequality and \eqref{eq: Var} yields the result.	
	\end{proof}
	\begin{lem}\label{Le: Var1}
		Assume $\Var(\xi_1) = \sigma^2 < \infty$. Then $V_{2,n}=o(1)$.
	\end{lem}
	\begin{proof}
		Arguing similarly as in the proof of Lemma~\ref{Le: Var}, given $\delta > 0$ and $\varepsilon > 0$, we have that $m$ large enough
		\begin{multline}\label{eq: Var1}
		\mathbb{P} \left( \left| \frac{1}{\sqrt{n}} \sum_{i = 1}^{n} \left[ \bar\xi_{i+1} \sum_{j=1}^{i}  \left( X_{i,j} - \frac{j^{\zeta - 1}}{p i^{\zeta}} \right) \right] \right| > \delta \right) \leq \\
		\mathbb{P} \left( \left| \frac{1}{\sqrt{n}} \sum_{i = 1}^{n} \left[ \bar\xi_{i+1}  \sum_{j=1}^{i}  \left( X_{i,j} - \frac{j^{\zeta - 1}}{p i^{\zeta}} \right)1_{A_{m,\varepsilon}} \right] \right| > \delta  
		\right) + \frac{\varepsilon}{2} 
		\end{multline}	
	and since $|X_{i,j}(\omega) - {j^{\zeta - 1}}/(p i^{\zeta}) |1_{A_{m,\varepsilon}} \leq (\varepsilon {j^{\zeta - 1}})/(p i^{\zeta})$
	for all $i \geq j \geq m$, we get that the first term on the right of~(\ref{eq: Var1}) is bounded above by
		\begin{multline*} 
		\frac{1}{\delta^2 n} \sum_{i=1}^{n}  \left[ \mathbb{E}(\bar\xi_{i+1})^2\, \mathbb{E}\!\left(  \sum_{j=1}^{i}  
		\left( X_{i,j}(\omega) - \frac{j^{\zeta - 1}}{p i^{\zeta}} \right)1_{A_{m,\varepsilon}} \right)^2  \right] \\
		\leq
		\frac{\sigma^2}{\delta^2 n} \sum_{i=1}^{n} \mathbb{E}\left[ \sum_{j=1}^{i} \left|  X_{i,j}(\omega) - \frac{j^{\zeta - 1}}{p i^{\zeta}} \right|1_{A_{m,\varepsilon}} \right]^2  
		\leq \varepsilon^2  \frac{\sigma^2}{\delta^2 n}  \sum_{i=1}^{n} \left( \sum_{j=m}^{i} \frac{j^{\zeta - 1}}{ i^{\zeta}} \right)^2 + o(1) \leq \frac{\varepsilon}{2},
		\end{multline*}
	as soon as $n$ is large enough, and the result follows upon substitution in \eqref{eq: Var1}, since $\delta$  and $\varepsilon$ are arbitrary.
	\end{proof}
	\begin{lem}\label{Le: Var2}
		Assume $\Var(\xi_1) = \sigma^2 < \infty$. Then $W_{4,n}=o(1)$.
	\end{lem}
	\begin{proof}
		Similar to the proof of Lemma \ref{Le: Var1}.
	\end{proof}	
	\begin{lem}\label{Le: Var3}
		Assume $\Var(\xi_1) = \sigma^2 < \infty$. Then $Z_{5,n}=o(1)$.
	\end{lem}	
	\begin{proof}	
		Changing the order of summation, we find that $Z_{5,n}$ equals constant times
		\begin{equation*}
		\frac{1}{\sqrt{n}} \sum_{k = 1}^{n} L_{k,n}\bar\eta_k\bar{M}_{k-1},
		\end{equation*}
	where 
	$L_{k,n}=\frac1{k^2}\left(\sum_{j=1}^kj^{\zeta-1}\right)\left(\sum_{i=k}^n\frac1{i^\zeta}\right)$, which is bounded above by constant times $\frac1 k$ uniformly in $j$ and $n$.
		Now by Markov: 
\begin{equation}
	\mathbb{P} \left( \left|Z_{5,n} \right| \geq \delta  \right) \leq 
	\frac{\mbox{const}}{\delta^2 n}\sum_{k=1}^{n}\frac1{k^2} \mathbb{E}(\bar{M}_{k-1}^2)\leq
	\frac{\mbox{const}}{\delta^2}\frac1 n\sum_{k=1}^{n}\frac1{k} 
	= o(1),
\end{equation}
		and we are done.
	\end{proof}	

	 We still owe a proof for Lemma \ref{Le: 2}.
	 
	 \begin{proof}[Proof of Lemma \ref{Le: 2}]
	 	Since $\bar{V}_i^2 - V_L^2 = (\bar{V}_i - V_L) (\bar{V}_i + V_L)$ and almost surely $\bar{V}_i$ converges to $V_L$, to prove the first claim is enough to show that $(\bar{V}_i^2 - V_L^2) = o(1/i^{1/2 - \epsilon})$.
	 	We write 
	 	\begin{eqnarray*}
	 		\bar{V}_i^2 - V_L^2 &=& 2F\sum_{j=1}^i\left[ \xi_j X_{i,j} \right] - \frac{2F\mu}{p}\int_0^1x^{\zeta -1} dx. \\
	 		&=& 
	 		{2F}  \sum_{j=1}^{i}  \left[\xi_j X_{i,j} - \mu \frac{j^{\zeta-1}}{p i^{\zeta}} \right] + 
	 		\frac{2F\mu}{p } \left[  \frac{1}{i} \sum_{j=1}^{i} \left( \frac{j}{ i} \right)^{\zeta - 1}- \int_{0}^{1} x^{\zeta - 1} dx  \right]. 
	 	\end{eqnarray*}
 	
	 	The second term on the right-hand side of this equation is an $O(1/i)$. We break down the first term as follows
	 	\begin{equation}\label{eq}
	 	2F \sum_{j=1}^{i} \bar\xi_j \frac{j^{\zeta -1}}{pi^{\zeta}} 
	 	+
	 	2F  \sum_{j=1}^{i}  \bar\xi_j \left(X_{i,j} - \frac{j^{\zeta - 1}}{p i^{\zeta}} \right) 
	 	+
	 	2F\mu\sum_{j=1}^{i} \left( X_{i,j} - \frac{j^{\zeta}}{pi^{\zeta}} \right).       
	 	\end{equation}
	 	Setting $\bar{S}_0 =0$ and $\bar{S}_k := \sum_{l=1}^k \bar\xi_k$, $k \in \mathbb{N}$, we write the first term on the right
	 	of~(\ref{eq}) as
	 	\begin{equation*}
	 	\sum_{j=1}^{i} \left[ (\bar{S}_j - \bar{S}_{j-1})\frac{j^{\zeta -1}}{pi^{\zeta}} \right] = \sum_{j=1}^{i-1} \left[ \bar{S}_{j}\left(\frac{j^{\zeta -1}}{pi^{\zeta}} - \frac{(j+1)^{\zeta -1}}{pi^{\zeta}}\right)\right] + \frac{\bar{S}_i}{pi}= o(1/i^{1/2 - \epsilon}),
	 	\end{equation*}
	 	where the last equality follows by the Law of the Iterated Logarithm. 
	 	
	 	Analogously, we write the second  term on the right of~(\ref{eq}) as
	 	\begin{equation}\label{eqa}
	 	\sum_{j=1}^{i-1} \left[ \bar{S}_{j} \left( X_{i,j} - X_{i,j+1}\right) \right] +  \sum_{j=1}^{i-1} \left[ \bar{S}_{j} \left(\frac{(j+1)^{\zeta -1}}{pi^{\zeta}} - \frac{j^{\zeta -1}}{pi^{\zeta}}\right)\right] + \bar{S}_i \left( X_{i,i} - \frac{1}{ip} \right). 
	 	\end{equation}
	 	Recalling \eqref{eq: X}, 
	 	one readily checks that that $|X_{i,j} - X_{i,j+1}| = O\left(|X_{i,j+1}|/{(M_j + 1)}\right)$. 
	 	Given $\varepsilon >0$, by Lemma \ref{Le: 1} we a.s.~find an $m \in \mathbb{N}$ such that $|X_{i,j+1}| \leq (1+c)(j+1)^{\zeta - 1}/(pi^{\zeta})$ for every $i \geq j \geq m$. Therefore, again by the Law of Large Numbers and the Law of the Iterated Logarithm, the three terms on~\eqref{eqa} are  $o(1/i^{1/2 - \epsilon})$.  
	 	
	 	To deal with the third and last term on the right of~(\ref{eq}), we may proceed similarly as in the analysis of $W_{2,n}$ above ---
	 	recall \eqref{eq: W}, \eqref{eq: Z}, \eqref{rij} and \eqref{eq: Z3}.
	     We write
	 	\begin{multline}\label{eqb}
	 	\sum_{j=1}^{i} \left[ X_{i,j} - \frac{j^{\zeta}}{pi^{\zeta}} \right]
	 	= 
	 	\sum_{j=1}^{i} \left[ \frac{j^{\zeta}}{i^{\zeta}}\left( \frac{1}{M_j} - \frac{1}{pj}\right)  \right] \\
	 	+ 
	 	\sum_{j=1}^{i} \left[ \frac{j^{\zeta}}{i^{\zeta}} \left( \frac{1}{M_j} - \frac{1}{pj}\right) \left(R_{i,j} + O(R_{i,j}^2)\right) \right] 
	 	+ 
	 	\sum_{j=1}^{i} \left[\frac{j^{\zeta - 1}}{pi^{\zeta}} \left(  R_{i,j} + O(R_{i,j}^2) \right) \right].
	 	\end{multline}
 	
	 	In the proof of Lemma \ref{Le: L3}, we have shown that almost surely, for $i \geq j$ sufficiently large, $|R_{i,j}| \leq 1/j^{1/2 - \epsilon}$, and we also argued that $\left( 1/M_j - 1/(pj) \right) = o(1/j^{3/2 - \epsilon})$. Using this estimates, we readly get that  
	 	each of the terms on the right hand side of \eqref{eqb} is an $o(1/i^{1/2 - \epsilon})$, for $0 < \epsilon < 1/4$, and thus, so is 
	 	the left hand side of \eqref{eqb}, and we are done with the first claim of the lemma.
	 	
	 	To argue the last claim of the lemma, note that
	 	\begin{eqnarray*}
	 		\frac{1}{\sqrt{n}} \sum_{i=1}^{n} \left[ \xi_i \left(\bar{V}_{i-1} - V_L\right)^2 \right] 
	 		&=& 
	 		o\left( \frac{1}{\sqrt{n}} \sum_{i=1}^{n} \frac{\xi_i}{i^{1 - 2\epsilon }}  \right) \nonumber \\
	 		&=&
	 		o\left( \frac{1}{\sqrt{n}} \sum_{i=1}^{n} \frac{\mu}{i^{1 - 2\epsilon }}  \right) 
	 		+
	 		o\left( \frac{1}{\sqrt{n}} \sum_{i=1}^{n} \frac{\bar\xi_i }{i^{1 - 2\epsilon }}  \right) = o(1/n^{1/2 - 2\epsilon}),
	 	\end{eqnarray*}  
 	where the last equality holds by the hypothesis that $\xi_1$ has finite second moment and the Two Series Theorem, and we are done. 
	 \end{proof}     
 
 Proceeding analogously as in the proof of Proposition~\ref{Pr: TCL1}, similarly breaking down the relevant quantities, 
 we may also obtain a central limit theorem for the velocity of the t.p.~on the modified process (at collision times), namely
 \begin{prop}\label{Pr: TCL2}
 	Let $\Var(\xi_1) = \sigma^2 < \infty.$ Then, as $n\to\infty$, 
 	\begin{equation*}
 	\sqrt{n}\left( \bar{V}_n - V_L\right) \Longrightarrow \mathcal{N}(0,\hat\sigma_v^2),
 	\end{equation*}
 where $\hat\sigma_v>0$.
 \end{prop}
	\section{Central Limit Theorem for the Original Process}
	
	
	In this section, we 
	prove our main results.
	
	\subsection{Proof of Theorem \ref{CLTQ}}\label{cltq}
	For each, $i \in \mathbb{N}$, let $t_i$ be the instant when the t.p.~collides for the first time with the initial $i-$th particle in the line; more precisely, $t_i$ is such that $Q(t_i) = S_i.$ It is enough to show a CLT along $(t_i)$, and for that it suffices to establish a version of Proposition \ref{Pr: TCL1} with barred quantities replaced by respective unbarred quantities, which ammounts to replacing $\bar t_n$ by $t_n$ in~(\ref{eq:clts}), namely showing that 
	$(S_n - {t}_nV_L)/\sqrt{n} \Longrightarrow \mathcal{N}(0,\hat\sigma_q^2)$.
	Theorem \ref{CLTQ} readily follows with $\sigma_q^2=\frac{V_L}\mu\hat\sigma_q^2$.
	
	We use Proposition \ref{Pr: TCL1} and a comparison between $\bar{t}_i$ and $t_i$ to conclude our proof.
	Due to Proposition \ref{Pr: TCL1}, it is enough to argue that
	\begin{equation}\label{eq: 41}
		\frac{t_n - \bar{t}_n}{\sqrt{n}} = o(1).	
	\end{equation}

	Let $s_1, s_2, \ldots$ be the instants when the t.p.~recollides with a moving elastic particle, whose velocities will be, respectively, denoted by $v_1, v_2, \ldots$. As follows from the remarks in the Introduction on the fact that the dynamics is a.s.~well defined --- see paragraph right below~\eqref{eqd} --- these sequences are well defined, and $s_1, s_2, \ldots$ has no limit points.
	We also recall that, for each $l \in \mathbb{N}$, $V(s_l)$ and $V(s_l^{+})$ denote the velocities of the t.p.~immediately before and at 
	the $l$-th recollision, respectively.
	 
	For each $j \in \mathbb{N}$ we define
	\begin{equation}\label{eq: Delta}
		\Delta(j) := \sum_{s_l \in [t_{j-1}, t_j] } \left[ V^2(s_l) - V^2(s_l^{+}) \right] ~~~ \text{ and } ~~~ \delta(j) := \sum_{s_l \in [t_{j-1}, t_j] } \left[ V(s_l) - v_l \right].
	\end{equation} 
	 As follows from what has been pointed out in the above paragraph, these sums are a.s.~well defined and consist of finitely many terms.

	Let $v: [0,\infty) \longrightarrow \mathbb{R}$ denote the function that associates the position $x$ to the velocity of the t.p.~at $x$, that is, $v(x) = V(Q^{-1}(x))$. We analogously define $\bar{v}: [0,\infty) \longrightarrow \mathbb{R}$ for the modified process. We have that
	\begin{equation*}
		t_n = \int_0^{S_n} \frac{1}{v(x)}~ dx ~~~ \text{ and } ~~~ \bar{t}_n = \int_{0}^{S_n} \frac{1}{\bar{v}(x)}~ dx.
	\end{equation*}
	In this way, \eqref{eq: 41} becomes
	\begin{equation*}
		\int_{0}^{S_n} \left( \frac{1}{v(x)} - \frac{1}{\bar{v}(x)} \right)~ dx = o(n^{1/2}),
	\end{equation*}
	and due to convergence of $v(x)$ and $\bar{v}(x)$, it is enough to argue that
	\begin{equation*}
		\int_{0}^{S_n}  \big( \bar{v}^2(x) - v^2(x) \big)~ dx = o(n^{1/2}).
	\end{equation*}
	
	Torricelli's equation, \eqref{eq:1}, \eqref{eq:2} and \eqref{eq: Delta}, give us that, for each $i \in \mathbb{N}$, at position $x \in [S_{i-1}, S_i)$,   
	\begin{equation}\label{eq: dif}
	\bar{v}^2(x) - v^2(x) = \sum_{j = 1}^{i-1} \left[ \Delta(j)  \prod_{k=j}^{i-1}\left( \frac{M_k + (\eta_k - 1)}{M_k + 1} \right)^2  \right] + \sum_{s_l \in [S_{i-1}, x)} \left( V^2(s_l) - V^2(s_l^{+})\right).
	\end{equation} 
	Therefore, we have the following upper bound
	\begin{equation}\label{eq: Up}
		\int_{0}^{S_n}  \big( \bar{v}^2(x) - v^2(x) \big)~ dx \leq \sum_{i = 1}^n \left[ \xi_i  \sum_{j = 1}^{i-1} \left( \Delta(j)  \prod_{k=j}^{i-1}\left( \frac{M_k + (\eta_k - 1)}{M_k + 1} \right)^2  \right) \right] + \sum_{i=1}^{n} \xi_i \Delta(i).
	\end{equation}
	
	Turning back to \eqref{eq:2}, we have that,
	\begin{eqnarray*}
	V(s_j) - V(s_j^+)  =
    V(s_j) - \left( \frac{M(s_j)-1}{M(s_j) + 1}V(s_j) + \frac{2}{M(s_j)+1}v_j \right) =  \frac{2(V(s_j) - v_j)}{M(s_j)+1}. 
	\end{eqnarray*}
	And therefore, again by the fact that $V(\cdot)$ is convergent, recalling \eqref{eq: Delta}, we have that
	$$\Delta(j) = O \left(\frac{\delta(j)}{M_j+1}\right);$$ 
	moreover, recalling 
	\eqref{eq: X}, we have  that
	\begin{multline}\label{eq: F}
		\sum_{i = 1}^n \left[ \xi_{i+1}  \sum_{j = 1}^{i} \left( \Delta(j)  \prod_{k=j}^{i}\left( \frac{M_k + (\eta_k - 1)}{M_k + 1} \right)^2  \right) \right] + \sum_{i=1}^{n} \xi_{i+1} \Delta(i+1) = \\
		O \left(  	\sum_{i = 1}^n \left[ \xi_{i+1} \sum_{j=1}^{i} \delta(j)X_{i,j} \right] + \sum_{i=1}^n \frac{\xi_{i+1} \delta(i+1)}{i+1}\right).
	\end{multline}
	
    By Lemma \ref{Le: 1}, 
	\begin{equation*}
	\sum_{i = 1}^n \left[ \xi_{i+1} \sum_{j=1}^{i} \delta(j)X_{i,j} \right] = \sum_{j=1}^n \left[ \delta(j) \sum_{i=j}^n \xi_{i+1}X_{i,j} \right] = O \left( \sum_{j=1}^n \left[ \delta(j)j^{\zeta - 1} \sum_{i=j}^n \frac{\xi_{i+1}}{i^{\zeta}} \right] \right).
	\end{equation*}
    Since $\mathbb{E}\xi^2 < \infty$, Borel-Cantelli lemma readily implies that for every $\epsilon > 0$, $\mathbb{P}(\xi_{n+1} > \epsilon \sqrt{n} ~~ \text{i.o.}) = 0$. Thus, 
	\begin{multline*} 
	 \sum_{j=1}^n \left[ \delta(j)j^{\zeta - 1} \sum_{i=j}^n \frac{\xi_{i+1}}{i^{\zeta}} \right] 
	 =
	 O\left(\sum_{j=1}^n \left[ \delta(j)j^{\zeta - 1} \sum_{i=j}^n \frac{\epsilon \sqrt{i}}{i^{\zeta}} \right] \right) 
	 = \\
	 \epsilon \sqrt{n} O\left( \  \sum_{j=1}^n \left[ \delta(j)j^{\zeta - 1} \sum_{i=j}^n \frac{1}{i^{\zeta}} \right] \right) 
	  =  \epsilon \sqrt{n} O\left(   \sum_{j=1}^n \delta(j) \right),		
	\end{multline*}
	and also 
	\begin{equation*}
	\sum_{i=1}^n \frac{\xi_{i+1} \delta(i+1)}{i+1} = O\left( \sum_{i=1}^n \delta(i+1) \right).
	\end{equation*}
	By Lemma~\ref{Le: finite}, we are done, 
	since $\epsilon > 0$ is arbitrary.
	\begin{lem}\label{Le: finite}
		Let $\delta(j)$ as defined in \eqref{eq: Delta}. Almost surely,
		\begin{equation}\label{eqe}
		\sum_{j=1}^{\infty} \delta(j) < \infty. 	
		\end{equation}
	\end{lem}
	\begin{proof}
		This result is already contained more or less explicitly in~\cite{FNV}, in the argument to prove Theorem \ref{Th: 1} --- see discussion on page 803 of~\cite{FNV}. For completeness and simplicity, circularity notwithstanding, we present an argument relying on
		Theorem \ref{Th: 1} directly. 
		
		There a.s.~exists a time $T_0$ such that 
		there are no recollisions with standing particles met by the t.p.~after $T_0$. This is because at large times, the velocity of the t.p.~is close enough to $V_L$ and its mass close enough to infinity, so that new collisions with standing elastic particles will give them velocity roughly $2V_L$, and thus they will be thence unreachable by the t.p.
        This means that we have only finitely many particles that recollide with the t.p.
		
		We may also conclude by an elementary reasoning using Theorem \ref{Th: 1} that if a particle collides infinitely often with the t.p., then its velocity may never exceed $V_L$. Let $u_1, u_2, \ldots$ denote the recollision times with such a particle, and $v(u_1), v(u_2), \ldots$, its velocity at such times, respectively. As we can deduce from \eqref{eq:2}, $v(u_{i+1}) > V(u_i)$; thus,
		\begin{equation}
		\sum_{i=1}^{\infty} \left[ V(u_i) - v(u_i) \right] < \sum_{i=1}^{\infty} \left[ v(u_{i+1}) - v(u_i) \right] \leq V_L,
		\end{equation}   
		and \eqref{eqe} follows. 
	\end{proof}
	
	\subsection{Proof of Theorem \ref{CLTV}}
	By Proposition \ref{Pr: TCL2}, and the convergences of both $V_n$ and $\bar{V}_n$, and after similar considerations as at the beginning of Subsection~\ref{cltq}, we find that it is enough to prove that 
	\begin{equation}\label{eqc}
		\sqrt{n}\left(\bar{V}^2_n - V^2_n \right) = o(1)
	\end{equation} 
	(so that in the end we get that Theorem \ref{CLTV} holds with $\sigma_v^2=\frac\mu{V_L}\hat\sigma_v^2$).
	
	Recalling \eqref{eq: dif}, we have that
	\begin{equation*}
		\bar{V}^2_n - V^2_n = \bar{v}^2(S_n) - v^2(S_n) = \sum_{j=1}^{n-1} \left[ \Delta(j)  \prod_{k=j}^{i-1}\left( \frac{M_k + (\eta_k - 1)}{M_k + 1} \right)^2 \right] + \Delta(n).
	\end{equation*}
	Proceeding similarly as in the proof of Theorem \ref{CLTQ}, we find that
	\begin{equation*}
	\sqrt{n}\left(\bar{V}^2_n - V^2_n \right) = O\left( \sqrt{n} \sum_{j=1}^{n} \left[ \delta(j)\frac{j^{\zeta-1}}{n^{\zeta}} \right] + \frac{\delta(n+1)}{\sqrt{n}} \right). 
	\end{equation*}
	By Lemma \ref{Le: finite}, given $\epsilon >0$, there almost surely exists $j_0 \in \mathbb{N}$ such that $\sum_{j \geq j_0} \delta(j) \leq \epsilon/2$. Thus,
	\begin{equation*}
	\sqrt{n} \sum_{j=1}^{n} \left[ \delta(j)\frac{j^{\zeta-1}}{n^{\zeta}} \right] \leq \frac{1}{n^{\zeta - 1/2}} \sum_{j=1}^{j_0} \delta(j)j^{\zeta -1} + \sum_{j > j_0}\delta(j) \leq \epsilon, 
	\end{equation*}
	for $n$ sufficiently large.
	Lemma \ref{Le: finite} implies that $\delta(n) = o(1)$. Since $\epsilon >0$ is arbitrary,~\eqref{eqc} follows.

\end{document}